\documentclass[
 reprint,
 amsmath,
 amssymb,
 aps,
 pre
]{revtex4-2}

\usepackage{hyperref}% add hypertext capabilities
\usepackage{graphicx}% Include figure files
\usepackage{dcolumn}% Align table columns on decimal point
\usepackage{bm}% bold math

\usepackage{soul}
\usepackage{bm}
\usepackage{fixmath}
\usepackage{xcolor}
\usepackage{amsthm}
\usepackage[ruled,linesnumbered]{algorithm2e}

\newtheorem{theorem}{Theorem}%  meant for continuous numbers
\begin{document}

\preprint{APS/123-QED}

\title{A distance function for stochastic matrices}

\author{Antony R. Lee}
\email{arl290@student.bham.ac.uk}
\author{Peter Ti\v{n}o}
\affiliation{School of Computer Science, University of Birmingham, Edgbaston, Birmingham, B15 2TT
}

\author{Iain B. Styles}
\affiliation{
 School of Electronics, Electrical Engineering, and Computer Science, Queen's University Belfast, University Road, Belfast, BT7 1NN
}

\date{\today}

\begin{abstract}
Motivated by information geometry, a distance function on the space of stochastic matrices is advocated. Starting with sequences of Markov chains the Bhattacharyya angle is advocated as the natural tool for comparing both short and long term Markov chain runs. Bounds on the convergence of the distance and mixing times are derived. Guided by the desire to compare different Markov chain models, especially in the setting of healthcare processes, a new distance function on the space of stochastic matrices is presented. It is a true distance measure which has a closed form and is efficient to implement for numerical evaluation. In the case of ergodic Markov chains, it is shown that considering either the Bhattacharyya angle on Markov sequences or the new stochastic matrix distance leads to the same distance between models.
\end{abstract}

%\keywords{Suggested keywords}%Use showkeys class option if keyword
                              %display desired
\maketitle

\section{Introduction}

Markov chains are a well studied tool in probability theory with numerous applications in communications~\cite{Omer2022ReviewSystems, AbuAlsheikh2015MarkovSurvey}, simulation~\cite{Roy2020ConvergenceCarlo}, queuing~\cite{Hamdi2021ANetworks} and many other areas. Of particular interest to us is the use of Markov chains to describe operational processes~\cite{Mor2021AApplications, vanZelst2021EventTaxonomy}, and in particular those applicable to general healthcare services~\cite{Martin2022DaQAPO:Assessment,Munoz-Gama2022ProcessChallenges, DeRoock2022ProcessArt.}. Of importance to population healthcare management are mental health~\cite{Claudio2023AInterventions}, community care~\cite{McClean2006UsingSystem}, and cost-effectiveness of interventions~\cite{Craig2002EstimationChain} which drive many decisions made by governments and professionals. In healthcare common questions about operational processes could include: What are the typical sequence of events when someone enters a hospital in an emergency? Can we redesign an operational process to be more efficient? How can we compare operational processes to see which are similar or different? In particular, the UK's National Institute for Health Excellence (NICE) regularly publishes guidance and recommendations based on estimating transition probabilities which model healthcare processes as a Markov chain~\cite{Srivastava2021EstimationAppraisals}, and for example see~\cite{NICE2020} for an explicit example of a transition matrix defined using clinical test results with respect to~\href{https://www.nice.org.uk/guidance/ta624/evidence/committee-papers-pdf-7081680925}{Peginterferon beta-1a for treating relapsing–remitting multiple sclerosis}. It is this focus on comparing descriptions of health services, represented as Markov type models, which motivates us here~\cite{Gatta2017PMineR:Medicine,Marazza2019ComparingCare,Ghahfarokhi2021ProcessCubes,Vallati2023OnStudy}. 

Within the domain of discrete time Markov chains, much work has been done on quantifying the similarity between two chains.  The most widely used approach to comparing Markov chains is to consider the probabilities induced on the sequence space they generate,  i.e comparing fixed length sequences of two models to distinguish different behaviour. Early work in this area \cite{Lyngs1999MetricsModels} established that vectorised (hidden) Markov chain transitions could be compared using simple functions: the dot product, cosine angle and Frobenius distance. Later, the use of the total variation distance on the space of traces generated by models was introduced~\cite{Zeng2010AModels}, followed by the formulation of ``linear'' distances between  samples from Markov chains~\cite{Daca2016LinearChains}. 
A general approach to the $L^{p}$ norm--induced distances and the Kullback-Leibler divergence has also been studied~\cite{Jaeger2014ContinuityProcesses}, extending them to infinite time runs while exploring their properties of continuity.

Albeit useful, the $L^{p}$ induced distances considered do not have a natural probabilistic  interpretation, and similarities based on the Kullback-Leibler divergence might be infinite or very sensitive to small changes to a stochastic matrix. More formally Kullback-Leibler--type divergences are not true distances because they sacrifice properties such as symmetry (quasimetric), identity of indiscernibles (pseudometric) or the triangle inequality (semimetric). This motivates the identification of a distance function for comparing the trace spaces of Markov chains which: 1) is a true distance, 2) efficient to evaluate, 3) incorporates the probabilistic nature of traces. To address these issues, we propose the use of the well known distance function, the Bhattacharyya angle, motivated and derived from first principles via information geometry by associating a Markov chain with a categorical random variable~\cite{Amari2001InformationDistributions, Amari2021InformationGeometry}. 

Although a step in a good direction, distances based on Markov sequences can in some scenarios still be overly sensitive to initial distributions or small changes in the stochastic matrix which represents a Markov chain. For example a simple model of a biased coin toss, where outcome bias is parameterised by some value $p\in[0,1]$, can generate completely difference trace sequences if comparing two different initial distributions. This is obvious, given two different starting conditions we end up at different places. However the underlying generative process, as specified by the stochastic matrix, is the same. 
In keeping with healthcare, another example of where the initial state is of less importance is for hospital emergency department patient attendances or other specific disease pathways. As the process typically starts in the same initial state, e.g. ``presented at hospital'', every comparison of processes across different organisations will use the same initial state distribution for a Markov chain.
It therefore seems reasonable to separate the initial conditions from the Markov chain's structural properties. This amounts to equipping the space of stochastic matrices with a distance function. This suggests an approach to comparing Markov chains via the structure of the stochastic matrices themselves, rather than traces generated by short or long term runs. 

The Markov theory literature for direct comparisons of stochastic matrices seems less formalised than its sequence space counterpart. Early work on comparing the structural influence of stochastic matrices can be found in the field of channel capacities~\cite{Cohen1998ComparisonsSciences}. In more recent work, Markov mixture models were created in~\cite{Zhou2021WhoModel} which use the log-likelihood function and BIC criteria to cluster a person's sequence of daily activity via estimated stochastic matrices. The investigation of the internal block structure of a special class of Markov chains in~\cite{Sanders2020ClusteringChains} used $L^{p}$ norms and the Kullback-Leibler divergence. They also appealed to Markov chain mixing times to show estimated values converge in a short enough time to be accurate. In the context of analysing protein structure models, \cite{Kawabata2000ProteinEvolution} propose a similarity score between the $n$ step stochastic matrices of two different proteins. Mixing long term and short term properties of a Markov chain, \cite{He2016FusingRecommendation} propose modelling recommendations for a user's next choice given previous choices as an order $k$ Markov chain. 

Learning a stochastic matrix, given observation data, is also important in the fields of sensor-based human activity recognition~\cite{Wang2020W-Trans:Recognition} and designing statistically consistent algorithms for noisy data labelling~\cite{Yao2020DualLearning}. Underlying these estimation techniques is the choice of objective function one might want to optimise. Typical objective functions used are: the direct calculation of the Kullback-Leibler divergence~\cite{Zou2016ReinforcementVehicle, Du2019IntelligentLearning} which is the usual divergence calculated row-wise and summed for a pair of stochastic matrices, some form of log-likelihood function as proposed for denoising in image analysis~\cite{Austin2021StructuredState-Spaces}, the related cross entropy for video analysis~\cite{Jabri2020Space-timeWalk}, cross entropy hypothesis testing an estimated stochastic matrix against a reference \cite{Nevat2018AnomalyTraffic} or even Bayesian based test approaches~\cite{Bacallado2009BayesianConstraint,Strelioff2007InferringModeling}. In the majority of these settings short term runs are more important as recommendations for the next action are all that is needed, or the long time run of a Markov chain gives no more benefit after a finite number of steps, and the information contained within the stochastic matrix is all that is needed. Our observation is that a well motivated comparison between Markov chain stochastic matrices would be beneficial, and could make more sense in certain scenarios than comparing sequences directly. This leads us to identify a distance function between stochastic matrices which mirrors desired properties of distances on sequence spaces. 

In summary, our ideas are to further develop the use of information geometry in Markov chain theory as a principled way to explore a true distance measure in two scenarios:
\begin{enumerate}
    \item Sequence space: Comparing the probability distributions induced by two Markov chains in terms of their $n$ step sequences
    \item Stochastic matrix space: Comparing the structural, i.e. local state to state jump, difference between two Markov chains, independently of initial distributions
\end{enumerate}

To address these ideas, our contributions in this work are:
\begin{enumerate}
    \item Motivation from first principles the use of the Bhattacharyya angle as an appropriate distance measure for comparing Markov chain sequences
    \item Extending these principles to a novel distance measure which compares Markov chain stochastic matrices directly
    \item Show how comparing sequences or stochastic matrices gives the same distance when considering ergodic Markov chains
\end{enumerate}

Throughout, we denote a time homogeneous discrete time Markov chain with $M$ which is defined over a state space $X=\lbrace x_{1},\ldots,x_{k}\rbrace$. In other words we consider sequences of random variables $(Z_{1}, Z_{2},\ldots)$ which satisfy the standard Markov property. Initial Markov chain distributions are written $\boldsymbol{\pi}$, with components $\pi(x_{i})$ representing the probability of being in state $x_{j}$, and the associated model's (row) stochastic matrix $\boldsymbol{P}$, with components $P(x_{i},x_{j}) = \mathrm{Pr}(Z_{t+1}=x_{j}\vert Z_{t}=x_{i})$. Sequences of length $n$ form a space $X^{n}$ and the space of infinite length sequences is written $X^{\omega}$. Over all elements of either $X^{n}$, or alternatively $X^{\omega}$, we can identify an induced probability vector $\boldsymbol{p}$, with components denoted $p(w)$, assigned for each sequence $w=(x_{1},\ldots,x_{n})\in X^{n}$, according to initial distribution $\boldsymbol{\pi}$ and stochastic matrix $\boldsymbol{P}$. The Iverson bracket is written as $[A]$ for some boolean statement $A$, and the collection of eigenvalues of a matrix $\boldsymbol{R}$ is denoted $\sigma(\boldsymbol{R})$. Finally we make use of the Hadamard product (i.e. element wise product) between two elements, denoted as $\circ$, and the Hadamard square root (i.e. element wise square root) denoted as $\circ\frac{1}{2}$. 

\section{Sequence distances}

Information geometry is the study of parametric families of probability distributions using the language of differential geometry~\cite{Amari2021InformationGeometry, Wolfer2023InformationSurvey}. An extremely well developed and diverse field, we will only present the topics we need and briefly. 

We treat the parameter space of a given statistical model as a manifold. The fundamental object in both information and differential geometry is the metric tensor. Manifolds allow us to discuss concepts such as continuity and differentiation. A metric tensor allows us to define a notion of distance between points on a manifold, which means we can equip the parameter space of a statistical model with a distance function. 

There are many ways to define a metric tensor for a manifold and historically in information geometry this has been via the Fisher information metric tensor. Consider a parametric family of stochastic models whose distribution function is denoted $h(x;\boldsymbol{p})$. Here $x$ is an element of the distribution's support space $X$ and $\boldsymbol{p}\in\mathbb{R}^{n}$ is a vector of the parameters for the family, which we will show can be treated as a probability itself. Up to a constant scalar, the Fisher information metric at a point $\boldsymbol{p}$ is defined as an $n \times n$ matrix with components, 
\begin{equation}\label{eqn:fisher_information_tensor_components}
    g_{ij}(\boldsymbol{p}) = -\mathbb{E}_{h}\left[\frac{\partial^{2}\log{h(x;\boldsymbol{p})}}{\partial p_{i} \partial p_{j}}\right]
\end{equation}
where $\mathbb{E}_{h}$ indicates the expectation value is taken with respect to the probability distribution $h(x;\boldsymbol{p})$ over the support of $x$. The $g_{ij}$ are known as metric tensor components, and define the geometry of the parameter space, which is a statistical manifold. We want to discuss the concept of geodesic curves, i.e. curves connecting two points of the parameter space which respect the underlying geometry of the space. Geodesics are derived using the geodesic equation which in turn is derived from the metric tensor. Intuitively, the solutions to the geodesic equations are the ``shortest'' path curves with respect to the geometry of the space under consideration. Conventionally, geodesic curves are parameterised by some real variable $s$, which encodes the start point and end point of the curve. Ultimately we want to find a closed form for the induced distance function, minimised over all possible geodesic curves $\boldsymbol{c}(s)$. In the case of a parametric family of probability distributions these points are denoted $\boldsymbol{c}(s_{1})=\boldsymbol{p}_{1}$ and $\boldsymbol{c}(s_{2})=\boldsymbol{p}_{2}$. In general, there can be multiple geodesic curves connecting two points. For example the two ways round a great circle on a sphere. Thus we define the distance between two points using the geodesic curve which has the smallest length,
\begin{equation}\label{eqn:minimum_length_distance}
    d(\boldsymbol{p}_{1},\boldsymbol{p}_{2}) = \min_{\boldsymbol{c}(s)}\int_{s_{1}}^{s_{2}}\mathrm{d}s\sqrt{\sum_{i,j}g_{ij}(\boldsymbol{p})\frac{\mathrm{d}c_{i}(s)}{\mathrm{d}s}\frac{\mathrm{d}c_{j}(s)}{\mathrm{d}s}}
\end{equation}
$d(\boldsymbol{p}_{1},\boldsymbol{p}_{2})$ would then be our desired distance function on a statistical manifold. It is well known that Markov chains induce a probability distribution over sequences of fixed length~\cite{Kiefer2018OnModels}. These distributions correspond to categorical distributions and are defined as having the probability mass function,
\begin{equation}
    h(x;\boldsymbol{p})=\prod_{i}p(x_{i})^{[x=x_{i}]}
\end{equation} 
with domain of input events $X=\lbrace x_{1},\ldots,x_{n}\rbrace$ and associated probabilities $\boldsymbol{p}=(p(x_{1}),\ldots,p(x_{n}))$. In other words the probability of observing each possible event $x_{i}$ is $p(x_{i})$, akin to the probability distribution of rolling an $n$-sided dice. For a fixed state space of $n$ elements, the space of all probabilities which can parametrise a categorical distribution is known as the simplex space $\Delta=\lbrace\boldsymbol{p}:\Vert\boldsymbol{p}\Vert_{1}=1\rbrace$. Thus from here we can compute the metric tensor components in Eq.~\ref{eqn:fisher_information_tensor_components} and work through until we arrive at a solution for Eq.~\ref{eqn:minimum_length_distance}. 

The metric tensor components we require are, up to an overall positive scaling constant,
\begin{equation}\label{eqn:fisher_metric_simplex}
    g_{ij}(\boldsymbol{p}) = \frac{\delta_{ij}}{p_{i}}
\end{equation}
where $\delta_{ij}$ is a Kronecker delta. In differential geometry a metric tensor can be used to define geodesic equations, i.e. equations which define in some sense the ``straightest'' path. Typically solving these equations is difficult, however in this case we can sidestep the usual challenge. The coordinate transformation $\boldsymbol{y}=\boldsymbol{p}^{\circ\frac{1}{2}}$ allows us to show the metric tensor components become $g_{ij}(\boldsymbol{y}) = \delta_{ij}$. This transformation shows the components are the standard Euclidean metric in Euclidean space, with the condition that points are constrained to a positive quadrant of unit hypersphere centered at the origin ($\boldsymbol{y}\cdot\boldsymbol{y}=1)$. Hence, the associated geodesic curves we want are arcs following the great circles defined by points $\boldsymbol{y}_{1}$ and $\boldsymbol{y}_{2}$. As a result, a closed form for the distance between two categorical distributions with $n$ states and parameterised with $\boldsymbol{p}_{1},\boldsymbol{p}_{2}\in \Delta$ is,
\begin{equation}
	d^{(n)}(\boldsymbol{p}_{1},\boldsymbol{p}_{2}) = 2\arccos\mathrm{BC}(\boldsymbol{p}_{1},\boldsymbol{p}_{2})
\end{equation}
This distance is known under various names such as the Rao distance~\cite{Nielsen2017ClusteringGeometry}, categorical Fisher-Rao distance~\cite{Miyamoto2024OnDistance}, but we will refer to it as the Bhattacharyya angle~\cite{Bhattacharyya1946OnDistributions}. Here we have made clear the link of this distance to the Bhattacharyya coefficient,
\begin{equation}
    \mathrm{BC}(\boldsymbol{p}_{1},\boldsymbol{p}_{2}) = \sum_{i=1}^{n}\sqrt{p_{1}(x_{i})p_{2}(x_{i})}
\end{equation}
The Bhattacharyya coefficient has been studied in relation to comparing discrete probability distributions already~\cite{Aherne1998TheData, Bi2019TheUpdating}. However the Bhattacharyya angle's use as a true distance measure does not yet seem to have been pursued for the purpose of comparing Markov chains.

The final link back to Markov chains comes from realising that the collection of all probabilities of observing an $n$ step Markov chain trace forms a categorical distribution over $X^{n}$. Explicitly, denoting a general sequence of events as $w = (x_{1}, x_{2}, \ldots, x_{n})$ for a chain $M$ and $n\ge2$, the probability of observing the trace $w$ given an initial distribution $\boldsymbol{\pi}$ for a Markov chain with stochastic matrix $\boldsymbol{P}$ is,
\begin{equation}
    p(w) = \pi(x_{1})P(x_{1},x_{2})\cdots P(x_{n-1},x_{n})
\end{equation}
Making the identification $M_{1}\rightarrow\boldsymbol{p}_{1}$, $M_{2}\rightarrow\boldsymbol{p}_{2}$ we arrive at a distance for Markov chains,
\begin{equation}
    d^{(n)}(M_{1},M_{2}) = 2\arccos\sum_{w\in X^{n}}\sqrt{p_{1}(w)p_{2}(w)}
\end{equation}
The key concepts here are: 1) using information geometry, a notion of distance can be naturally motivated for Markov chains, 2) the distance satisfies all properties of a true distance function and 3) the distance is constructed to respect the basic probability properties of a Markov chain. Following~\cite{Kazakos1978TheChains}, for convenience we introduce the following definitions, where $\circ$ denote Hadamard operations, 
\begin{subequations}
    \begin{align}
        \boldsymbol{r} &= \boldsymbol{\pi}_{1}^{\circ \frac{1}{2}}\circ\boldsymbol{\pi}_{2}^{\circ \frac{1}{2}} \\
        \boldsymbol{R} &= \boldsymbol{P}_{1}^{\circ\frac{1}{2}}\circ\boldsymbol{P}_{2}^{\circ\frac{1}{2}}
    \end{align}
\end{subequations}
Considering the product structure for each $p_{i}(w)$ and denoting the vector of all ones as $\textbf{1}$, we can compactly write the sum over length $n$ words,
\begin{equation}\label{eqn:compact_ba}
    % d^{(n)}(M_{1},M_{2}) = 2\arccos\left\Vert\boldsymbol{r}^{T}\boldsymbol{R}^{n}\right\Vert_{1}
    d^{(n)}(M_{1},M_{2}) = 2\arccos\left(\boldsymbol{r}^{T}\boldsymbol{R}^{n}\textbf{1}\right)
\end{equation}
The Bhattacharyya angle in Eq.~\ref{eqn:compact_ba} serves as a true distance measure for comparing Markov chain induced probabilities. It is numerically very appealing as calculations can be reduced to matrix and vector multiplication. 

Turning to long run sequences, we can compute asymptotic quantities akin to the Kullback-Leibler divergence rate. Such quantities are important in the analysis of systems which have become stationary, i.e. do not change anymore, and are linked to the information content of a probability distribution. Typically they compare infinite step traces for two Markov chains, averaged by the number of steps. However we immediately see for any pair of models $d^{(n)}(M_{1},M_{2})/n \rightarrow 0$ as $n\rightarrow\infty$. This is similar to most ``distances'' considered in Markov theory, with the notable exception of the Kullback-Leibler divergence rate~\cite{Rached2004TheSources}. Hence we are motivated to study the non-regularised distance limit. In the case of the Bhattacharyya angle the quantity we are interested in is,
\begin{equation}
    \begin{aligned}
        d^{\omega}(M_{1},M_{2}) &= \lim_{n\rightarrow\infty}d^{(n)}(M_{1},M_{2}) \\
        & = 2\arccos\sum_{w\in X^{\omega}}\sqrt{p_{1}(w)p_{2}(w)}
    \end{aligned}
\end{equation}
Note this is equivalent to the matrix product form of $d^{\omega}(M_{1},M_{2}) = \lim_{n\rightarrow\infty}2\arccos\left(\boldsymbol{r}^{T}\boldsymbol{R}^{n}\textbf{1}\right)$. We can see the Bhattacharyya angle always exists in the limit $n\rightarrow\infty$ by observing (via the inequality of arithmetic and geometric means) $\sqrt{p_{1}(w)p_{2}(w)}\le \frac{p_{1}(w)+p_{2}(w)}{2}$ and using the direct comparison test for series convergence~\cite{Munem1984CalculusMunem}. As $\sum_{w\in X^{\omega}}\frac{p_{1}(w)+p_{2}(w)}{2}=1$ (i.e. converges) the direct comparison test proves $d^{\omega}(M_{1},M_{2})$ converges. Convergence of the Bhattacharyya angle is then ensured by the continuity of $\arccos$ on $[0,1]$. 

A simple characterisation of when $d^{\omega}(M_{1},M_{2})$ is guaranteed to attain its maximum can be found from the inequality $\rho(\boldsymbol{P}_{1}^{\circ\frac{1}{2}}\circ\boldsymbol{P}_{2}^{\circ\frac{1}{2}})\le\rho(\boldsymbol{P}_{1})\rho(\boldsymbol{P}_{2})=1$ (see Huang et. al. for details~\cite{Huang2011OnMatrices}). In other words when the spectral radius of $\boldsymbol{P}_{1}^{\circ\frac{1}{2}}\circ\boldsymbol{P}_{2}^{\circ\frac{1}{2}}$ is strictly less than unity the distance converges to the numerical constant $\pi$ in the limit $n\rightarrow\infty$. This means the long term run of two such Markov chains are statistically distinguishable according to the Bhattacharyya angle. We can however go further and study two types of matrices for which a closed form for the ``Bhattacharyya rate'' $d^{\omega}(M_{1},M_{2})$ is immediately computable and not trivially a maximum~\cite{Kazakos1978TheChains},
\begin{enumerate}
    \item Type 1: Diagonalisable matrices (over $\mathbb{C}$) with real eigenvalues
    \item Type 2: Primitive matrices %Non-negative and irreducible matrices
\end{enumerate}
Primitive matrices are defined as non-negative matrices (i.e. all elements are non-negative) such that for some positive integer $k>0$ the matrix power $\boldsymbol{R}^{k}$ is a strictly positive matrix (i.e. all elements are positive). Primitive matrices play a central role in various proofs based on the Perron-Frobenius theorem~\cite{Dembele2021AMatrices}. Although with substantial overlap, these matrix types do have non-trivial complements in the sense that there are matrices which exist as one type and not the other. 

% we can collect the various results as follows, 
% \begin{equation}
%     d^{\omega}(M_{1}M_{2}) = 
%         \begin{cases} 2\arccos\sum\limits_{i\in\mathcal{I}}\Vert\boldsymbol{r}^{T}\boldsymbol{p}_{i}\boldsymbol{q}_{i}^{T}\Vert_{1} & \text{: type 1} \\
%         2\arccos\Vert\boldsymbol{r}^{T}\boldsymbol{p}_{1}\boldsymbol{q}_{1}^{T}\Vert_{1} & \text{: type 2} \\
%         \pi & \text{: $\rho(\boldsymbol{R})<1$}
%         \end{cases}
% \end{equation} 

\begin{theorem}[Bhattacharyya rate]\label{thm1}
The Bhattacharyya rate between two Markov chain models which satisfy type 1 or type 2 requirements is
\begin{equation}
    d^{\omega}(M_{1}, M_{2}) = 
        \begin{cases} 2\arccos\sum\limits_{i\in\mathcal{I}}\boldsymbol{r}^{T}\boldsymbol{p}_{i}\boldsymbol{q}_{i}^{T}\boldsymbol{1} & \mathrm{:type~1} \\
        2\arccos\left(\boldsymbol{r}^{T}\boldsymbol{p}_{1}\boldsymbol{q}_{1}^{T}\boldsymbol{1}\right) & \mathrm{:type~2} \\
        \pi & \mathrm{:~}\rho(\boldsymbol{R})<1
        \end{cases}
\end{equation} 
\end{theorem}
\begin{proof}
Define the requirements of the two types of Markov chains as
\begin{enumerate}
    \item Type 1: $\boldsymbol{R}$ is diagonalisable, $\rho(\boldsymbol{R})=1$, $\lambda_{i}>-1$, and $\mathcal{I}=\lbrace i:\lambda_{i}=1,\lambda_{i}\in\sigma(\boldsymbol{R})\rbrace$
    \item Type 2: $\boldsymbol{R}$ is primitive, $\rho(\boldsymbol{R})=\lambda_{1}=1$
\end{enumerate}
Denote the eigenvalues, right eigenvectors and left eigenvectors of $\boldsymbol{R}$ as $\lambda_{i}$, $\boldsymbol{p}_{i}$ and $\boldsymbol{q}_{i}$ respectively. The definition of diagonalisation allows us to write type 1 matrices as $\boldsymbol{R}=\sum_{i}\lambda_{i}\boldsymbol{p}_{i}\boldsymbol{q}_{i}^{T}$. Thus we can write an arbitrary matrix power as $\boldsymbol{R}^{n}=\sum_{i}\lambda_{i}^{n}\boldsymbol{p}_{i}\boldsymbol{q}_{i}^{T}$. The eigenvalues with absolute value strictly less than one vanish in the limit $n\rightarrow\infty$. Eigenvalues equal to unity are the only ones which remain. In a similar way, the Perron-Frobenius theorem states any type 2 matrix (i.e. primitive matrices) has the limit $\boldsymbol{R}^{n}\rightarrow\boldsymbol{p}_{1}\boldsymbol{q}_{1}^{T}$ as $n\rightarrow\infty$. Therefore we can collect the results as
\begin{equation}
    \lim_{n\rightarrow\infty}\boldsymbol{R}^{n} = 
        \begin{cases} 
        \sum_{i\in\mathcal{I}}\boldsymbol{p}_{i}\boldsymbol{q}_{i}^{T} & \text{: type 1} \\
        \boldsymbol{p}_{1}\boldsymbol{q}_{1}^{T} & \text{: type 2}
        \end{cases}
\end{equation}
Inserting the above into Eq.~\ref{eqn:compact_ba} we obtain a closed form for $d^{\omega}(M_{1}, M_{2})$ and arrive at our result. 
\end{proof}
The closed form nature of the Bhattacharyya angle is such that numerical analysis can be efficiently implemented, which suggests we look at a position between short term runs and comparing Markov chains asymptotically through the field of mixing times~\cite{Montenegro2006MathematicalChains,Dyer2006MarkovComparison}. 

Markov chain mixing times are concerned with estimating a finite time step $\tau_{\star}$ such that the distance between some reference stationary distribution and the $\tau_{\star}$ step distribution of some Markov chain is ``close''. Denoting an initial Markov chain distribution for a given state $x_{j}\in X$ as $\boldsymbol{e}_{j}$ (i.e. the standard basis in $\mathbb{R}^{n}$), we define the Bhattacharyya angle induced mixing time of a Markov chain $\boldsymbol{P}$ (with dimensions $n\times n$) starting in state $\boldsymbol{e}_{j}$, with unique stationary distribution $\boldsymbol{\pi}_{\star}$, as
\begin{equation}
    \tau_{\star}(x)=\min\left\lbrace \tau>0 : d^{(n)}\left(\boldsymbol{\pi}_{\star}, \boldsymbol{e}_{j}^{T}\boldsymbol{P}^{\tau}\right) \leq \epsilon \right\rbrace
\end{equation}
We can operationally call two Markov chains ``similar'' if after some time $\tau_{\star}$ their state distributions are ``close enough'' for the problem at hand. Another way to view this is that if we have a obtained a mixing time for one chain, how can we use it to bound the mixing time of another chain? In some sense one can then approximate one Markov chain with another after a finite number of steps, when comparing the marginal state distributions of two Markov chains. 

A common class of Markov chains considered in mixing times are ergodic (irreducible and aperiodic) and reversible. Ergodic means that a discrete finite Markov chain is,
\begin{enumerate}
    \item Irreducible: $\forall i,j\,\exists\,\tau \textrm{ such that } P^{(n)}_{ij}>0$, i.e. all states are reachable from another in a finite number of steps $\tau$
    \item Aperiodic: The greatest common divisor of the lengths of all cycles of the Markov chain's stochastic matrix is one
\end{enumerate}
and a well known result is that ergodic chains have a unique stationary distribution $\boldsymbol{\pi}$. For an ergodic chain, reversible means the associated unique stationary distribution and stochastic matrix satisfy,
\begin{equation}
    \mathrm{diag}(\boldsymbol{\pi})\,\boldsymbol{P}=\boldsymbol{P}^{T}\,\mathrm{diag}(\boldsymbol{\pi})
\end{equation}
which are known as the detailed balance equations. Under these assumptions, bounds on $\tau_{\star}$ can be placed as we can take advantage of the fact that the $\tau$-step transition matrix components can be written as~\cite{Dyer2006MarkovComparison},
\begin{equation}\label{eqn:ergodic_reversible_markov_powers}
P^{\tau}(x_{j},x_{k}) = \pi_{\star}(x_{k}) +\sqrt{ \frac{\pi_{\star}(x_{k})}{\pi_{\star}(x_{j})}}\sum_{i=2}^{n}\lambda_{i}^{\tau}e^{(i)}_{j}e^{(i)}_{k}
\end{equation}
Here $\lambda_{i}$ are the eigenvalues of $\boldsymbol{P}$, the $\boldsymbol{e}^{(i)}$ are an orthonormal basis $\boldsymbol{e}^{(i)}\cdot\boldsymbol{e}^{(j)}=\delta_{ij}$, with $\boldsymbol{e}^{(1)}=\boldsymbol{\pi}_{\star}^{\circ\frac{1}{2}}$. Finally the assumption $\boldsymbol{\pi}_{\star}>\boldsymbol{0}$ for the unique stationary distribution of $\boldsymbol{P}$ is imposed. Usefully, the eigenvalues of ergodic and reversible Markov chains are real and satisfy $\lambda_{1}=1>\lambda_{2}\ge\ldots\ge\lambda_{n}>-1$. This allows us to take the limit $\tau\rightarrow\infty$ and know with certainty the power of all eigenvalues, except $\lambda_{1}$, will eventually vanish. With the expression in Eq.~\ref{eqn:ergodic_reversible_markov_powers}, and using the Bhattacharyya angle's Taylor series, with associated bounds on remainder terms, we can give two conditions needed to be sufficient to bound the mixing time,
\begin{subequations}
    \begin{align}
        \tau_{-} &\ge\frac{1}{\log\lambda_{\max}}\log\gamma\left(1+\frac{1}{\pi_{-}}\right)^{-1} \\
        \tau_{+} &\ge \frac{1}{2\log\lambda_{\mathrm{max}}}\log 4(1-\cos\epsilon/2)\left(1+\frac{1}{\pi_{-}}\right)^{-1}
    \end{align}
\end{subequations}
where we have introduced the so-called spectral gap $\lambda_{\max} = \max(\lambda_{2},\vert\lambda_{n}\vert)$, with $\pi_{-} = \min\boldsymbol{\pi}_{\star}$, and the constant $\gamma=1-(5/16)^{2/7}\approx0.282$. The constant $\gamma$ is chosen to ensure the remainder of the Taylor series expansion of the mixing time can be replaced safely with a known lower bound, which entails a minimum requirement on $\tau_{\star}$ given by $\tau_{-}$. The second inequality using $\tau_{+}$ gives a sufficient condition which guarantees the desired distance accuracy $\epsilon$ has been achieved. Consequently, the true mixing time is known to fall within the range $[\tau_{-}, \tau_{+}]$. Given the number of uses, and easily obtained results afforded by the Bhattacharyya angle we promote it as a robust tool to investigate both short and long time Markov chain runs. 

\section{Stochastic matrix distances}

Our idea is to, again appealing to information geometry, elevate the Bhattacharyya angle from a probability space to matrices. To this end, note that any stochastic matrix is a collection of individual state-conditional next-state probability distributions~\cite{Lebanon2005AxiomaticModels, Montufar2014OnPolytopes}. This can be interpreted as a stochastic matrix being an element of a Cartesian product space of underlying categorical distributions (as explained in the trace distance section). Thus we identify $\boldsymbol{P}=(\boldsymbol{p}_{1}, \ldots, \boldsymbol{p}_{n})\in\Delta\times\cdots\times\Delta$, where $\Delta$ is the $n$ dimensional simplex space, i.e. $\boldsymbol{P}$ is an $n\times n$ matrix. Such a decomposition induces what is known as a product metric in differential geometry~\cite{Miyamoto2024OnDistance}. In other words, we can build a distance function on the space of stochastic matrices by using the Bhattacharyya angle inherited from the underlying simplex space $\Delta$. Using a product metric structure here is actually very natural. The work of~\cite{Lebanon2005AxiomaticModels, Montufar2014OnPolytopes} characterise the metric tensor components for stochastic matrices in terms of invariance under linear transformations called Markov maps. The result is choosing a product metric structure or asking for invariance under Markov maps gives the same distance on stochastic matrices.

As a demonstration using two manifolds, consider the product metric tensor induced via the Cartesian product $(\Delta,g)\times (\Delta,g)$ where $g$ is the metric tensor defined via Eq.~\ref{eqn:fisher_metric_simplex}. The product metric tensor components in this case are defined as,
\begin{equation}
    g^{\mathrm{prod}}_{ij}\left(\boldsymbol{p}_{1},\boldsymbol{p}_{2}\right) = g_{ij}\left(\boldsymbol{p}_{1}\right) \oplus g_{ij}\left(\boldsymbol{p}_{2}\right)
\end{equation}
Further, we can construct the product curve which connects pairs of points on $\Delta\times\Delta$ as $\boldsymbol{c}(t) = \left(\boldsymbol{c}_{1}(t),\boldsymbol{c}_{2}(t)\right)$ where $\boldsymbol{c}_{1}(t)$ and $\boldsymbol{c}_{2}(t)$ are the geodesic curves on the two manifolds respectively. A well known result in differential geometry is that product metrics induce an associated distance measure on the product manifold,
\begin{equation}
    d_{\mathrm{prod}} \left(\boldsymbol{p}_{1}\oplus\boldsymbol{q}_{1},\boldsymbol{p}_{2}\oplus\boldsymbol{q}_{2}\right) =\sqrt{d^{2}_{1}\left(\boldsymbol{p}_{1},\boldsymbol{p}_{2}\right)+d^{2}_{2}\left(\boldsymbol{q}_{1},\boldsymbol{q}_{2}\right)}
\end{equation}
More generally, for an $n$ dimensional simplex we have the general form for a distance measure as $d_{\mathrm{prod}}=\sqrt{\sum_{i=1}^{n}d^{2}_{i}}$ and, consequently, we can define a distance function on the space of stochastic matrices.
\begin{theorem}[Stochastic Matrix Distance]\label{thm2}
Let $\boldsymbol{P}_{1}$, $\boldsymbol{P}_{2}$ be square (row) stochastic matrices. The minimum length geodesic distance between $\boldsymbol{P}_{1}$, $\boldsymbol{P}_{2}$ is
\begin{equation}\label{eqn:smd}
        d\left(\boldsymbol{P}_{1}, \boldsymbol{P}_{2}\right) = 
         2\sqrt{\sum_{x_{i}\in X}\arccos^{2}\sum_{x_{j}\in X}\sqrt{P_{1}(x_{i},x_{j})P_{2}(x_{i},x_{j})}}
\end{equation}
\end{theorem}
The proof of theorem~\ref{thm2} can be found in appendix~\ref{proof:thm2}. As well as enjoying many of the previously defined properties of the Bhattacharyya angle, it is also valid for any pair of stochastic matrices. There are no other requirements to compute the distance on the space of stochastic matrices. This appears to be a genuinely new distance function derived via information geometry. Although motivated by its use for short time Markov chains, we can again consider infinite step distances for ergodic Markov chains and can actually prove the distance reduces to the Bhattacharyya angle between the stationary distributions of the two chains,
\begin{theorem}[Ergodic distance]\label{thm3}
    Let $M_{1}$, $M_{2}$ be two ergodic discrete time Markov chains with unique stationary distributions $\boldsymbol{\pi}_{1}$, $\boldsymbol{\pi}_{2}$ respectively. The ergodic distance between $M_{1}$ and $M_{2}$ is 
    \begin{equation}\label{eqn:ergodic_asymp_distance}
        d_{E}(M_{1},M_{2}) = 2\arccos\mathrm{BC}(\boldsymbol{\pi}_{1},\boldsymbol{\pi}_{2})
    \end{equation}
\end{theorem}
\begin{proof}
    In the case of a ergodic (i.e. irreducible and aperiodic) Markov chain $\boldsymbol{P}_{1}$, with stationary distribution $\boldsymbol{\pi}_{1}$, via the Perron-Frobenius theorem the matrix power of the model's stochastic matrix has the limit,
    \begin{equation}
        \lim_{N\rightarrow\infty}\boldsymbol{P}_{1}^{N} = \textbf{1}^{T}\boldsymbol{\pi}_{1}
    \end{equation}
    Such a limit is just a statement that the matrix power converges to a matrix where each row is a copy of the stationary distribution. From here it is straightforward to show for the two models 
    \begin{equation}
        \sum_{x_{j}\in X}\sqrt{P_{1}(x_{i},x_{j})P_{2}(x_{i},x_{j})} = \mathrm{BC}(\boldsymbol{\pi}_{1},\boldsymbol{\pi}_{2})
    \end{equation}
    As the above is independent of $x_{i}$, the first summation inEquation~\ref{eqn:smd} reduces to
    \begin{equation}
        |X|\arccos^{2}\mathrm{BC}(\boldsymbol{\pi}_{1},\boldsymbol{\pi}_{2})
    \end{equation}
    For Markov chains with finite states we have $|X|=n$ and thus obtain,
    \begin{equation}
        \lim_{N\rightarrow\infty}d\left(\boldsymbol{P}_{1}^{N},\boldsymbol{P}_{2}^{N}\right) = 2\sqrt{n}\arccos\mathrm{BC}(\boldsymbol{\pi}_{1},\boldsymbol{\pi}_{2})
    \end{equation}
    Finally, notice that $\lim_{N\rightarrow\infty}^{n\rightarrow\infty} d\left(\boldsymbol{P}_{1}^{N},\boldsymbol{P}_{2}^{N}\right)/\sqrt{n}$ is finite, independent of the value of $n$ and thus equal to the Bhattacharyya angle between the stationary distributions of the Markov chains. 
    % This motivates the definition of a general distance function for ergodic Markov chains as,
\end{proof}
Intuitively this makes sense as the long term run of an ergodic Markov chain converges to its stationary distribution. Therefore any differences between two Markov chains should only be accounted for by their stationary distributions, and not initial conditions. As noted in the mixing times section, the stochastic matrix distance for ergodic chains is also identical to considering the distance between the asymptotic distributions from starting in any arbitrary initial distribution, i.e. $d^{(n)}(\boldsymbol{e}_{x}^{T}\boldsymbol{P}^{N}_{1},\boldsymbol{e}_{y}^{T}\boldsymbol{P}^{N}_{2}) \rightarrow d_{E}(M_{1},M_{2})$ as $N\rightarrow\infty$. This adds further weight to using the Bhattacharyya angle and stochastic matrix distance as they have a natural correspondence on ergodic chains.  

\section{Numerical Simulation}

To illustrate the use of the new stochastic matrix distance, and compare it alongside other typical distance functions, we simulate the random generation of four clusters of stochastic matrices. For each cluster, we generate a set of $3\times 3$ stochastic matrices where each row is drawn from a Dirichlet distribution with parameter vector $\boldsymbol{\alpha}=(\alpha_{1},\alpha_{2},\alpha_{3})$. The vector $\boldsymbol{\alpha}$ is a strictly positive and real vector, and controls the centre and spread of the corresponding Dirichlet distribution. This setup means we can tune the centres and spread of each cluster while making each stochastic matrix ``similar'' in both row values and overall matrix values. Initially we fix a distinguished cluster as a reference and denote this with Dirichlet parameters $\boldsymbol{\alpha}_{0}$. The remaining three clusters are generated with distinct parameters $\boldsymbol{\alpha}_{i}$. Finally, we increment along the transformation $\boldsymbol{\alpha}_{i}(t)=(1-t)\boldsymbol{\alpha}_{i}+t\boldsymbol{\alpha}_{0}$ for different values of $t\in[0,1]$, observing the accuracy of clustering for various distance functions. A pseudo-procedure for the simulation is laid out in~\ref{proc:simulation}. We choose the (adjusted) Rand index as the figure of merit for the clustering accuracy and average the index over many runs to mitigate observing a single extreme result~\cite{Steinley2004PropertiesIndex.}. Intuitively, the Rand index measures agreement, across two given clusterings, by checking how many pairs of clustered data elements are in the same cluster, and which pairs of clustered data elements are not in the same cluster. Ideally, all pairings according to one clustering match all parings in the other cluster, giving complete agreement. In practice it is better to adjust the Rand index under the null hypothesis assumption the clusterings were randomly assigned with equal probability by permuting the data elements, hence giving the adjusted Rand index as a measure of agreement.

Formally, given $N$ data points $X = \lbrace x_{i}|i\in[N]\rbrace$, for two clusterings $U=\lbrace U_{1},\ldots,U_{R}\rbrace$ and $V=\lbrace V_{1},\ldots,V_{C}\rbrace$ we can quantify the two clusterings contingency as $n_{ij}=|U_{i}\cap V_{j}|$. The collection of integers $n_{ij}$ count the overlap between different clusters. In consequence, we define the auxiliary quantities,

\begin{subequations}
    \begin{gather}
        n_{i+} = \sum_{j=1}^{C}n_{ij},\quad 
        n_{+j} = \sum_{i=1}^{R}n_{ij}
    \end{gather}
    \begin{gather}
        \bar{t} = \sum_{i=1}^{R}\sum_{j=1}^{C}n^{2}_{ij},\quad
        \bar{r} = \sum_{i=1}^{R}n^{2}_{i+},\quad
        \bar{c} = \sum_{j=1}^{C}n^{2}_{+j}
    \end{gather}
    \begin{gather}
        \begin{gathered}
            a = \frac{\bar{t}-N}{2},\quad
            b = \frac{\bar{r}-\bar{t}}{2},\quad
            c = \frac{\bar{c}-\bar{t}}{2},\\
            d = \frac{\bar{t}-\bar{r}-\bar{c}+N^{2}}{2}
        \end{gathered}       
    \end{gather}
\end{subequations}
Taken together, these auxiliary quantities represent the agreement between two different clusterings of the data $X$. We can collect the above into a matrix $\boldsymbol{M}$ and express the adjusted Rand index in a compact manner as,

\begin{equation} \label{eq:two_equations}
  \begin{aligned}
    \boldsymbol{M} = 
    \begin{pmatrix}
        a & b \\
        c & d
    \end{pmatrix}, \quad & 
    \boldsymbol{P} = 
    \begin{pmatrix}
        0 & 1 \\
        1 & 0
    \end{pmatrix}
  \end{aligned}
\end{equation}

\begin{equation}
    \mathrm{Rand_{adjusted}} = \frac{\binom{N}{2}\mathrm{tr}\boldsymbol{M}-\mathrm{tr}\boldsymbol{P}\boldsymbol{M}^{2}}{\binom{N}{2}^{2}-\mathrm{tr}\boldsymbol{P}\boldsymbol{M}^{2}}
\end{equation}
Where $\boldsymbol{P}$ is known as the $2\times 2$ exchange matrix (i.e. it exchanges rows and columns). We compared the symmetric Kullback-Leibler divergence rate, matrix distances induced by the $L^{1}$ and $L^{2}$ norms, and our new stochastic matrix distances. We use cluster Dirichlet parameters, visualised in Fig.~\ref{fig:ternary_plot}, $\boldsymbol{\alpha}_{0} = (8,8,8)$, $\boldsymbol{\alpha}_{1} = (8,2,2)$, $\boldsymbol{\alpha}_{2} = (2,8,2)$, $\boldsymbol{\alpha}_{3} = (2,2,8)$. 

The results of the simulation are shown in Fig.~\ref{fig:rand_index}, where a value of one on the vertical axis indicates clusters have been assigned correctly, and lower values mean cluster memberships are incorrect. We observe that the metrics all have a similar behaviour. Moreover, the stochastic matrix distance performance is comparable to the other functions. All functions approach a mean adjusted Rand score of zero as the distributions for clusters converge i.e. the clusters are indistinguishable and hence cluster membership can not be determined. We stress that the purpose of this simulation is not to pick a ``best'' distance to perform clustering, as this is too subjective, but to show our new distance function indeed behaves sensibly in a controlled setting.

\begin{figure}
    \centering
    \includegraphics[width=1\linewidth]{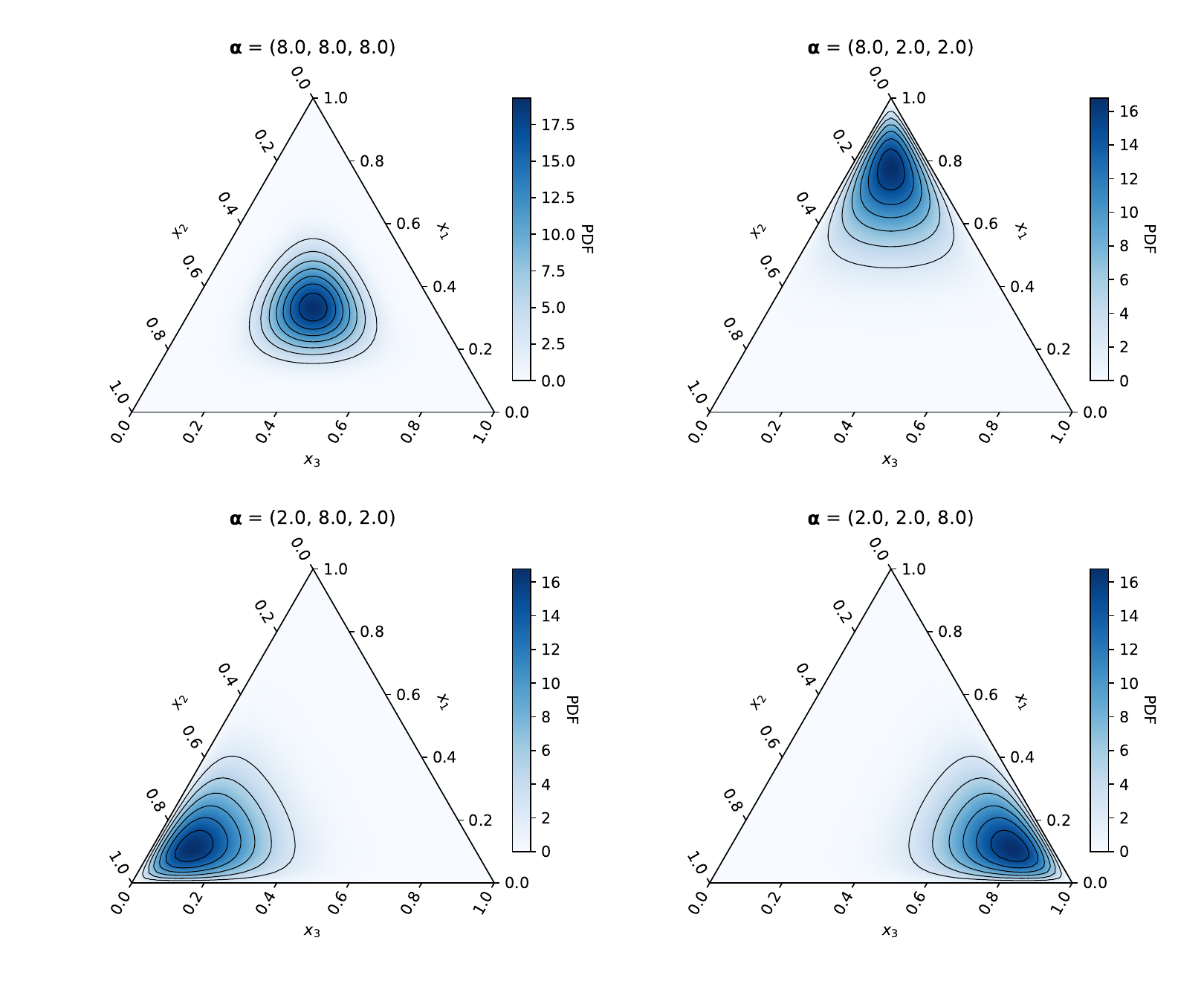}
    \caption{Ternary plot of the four base Dirichlet distributions used to generate clusters of stochastic matrices. The top-left density plot serves as a reference distribution. The parameters $\boldsymbol{\alpha}_{i}(t)$ approach $\boldsymbol{\alpha}_{0}$ linearly such that all distributions eventually coincide, and thus cannot be distinguished.}
    \label{fig:ternary_plot}
\end{figure}

Applications to real world data and settings will be a better test of the new stochastic matrix distance. We do emphasise though the stochastic matrix distance is a true distance measure, purposefully motivated through information geometry and so seems a natural distance measure to pick in the sense of probability theory. In terms of machine learning, both the sequence space and stochastic matrix space distances put forward here seem to be more principled as loss functions for categorical probability models. 

\begin{algorithm}[htbp] % Use [H] for 'here', or [htbp] for float placement
  \DontPrintSemicolon 
  \SetKwInOut{KwRequire}{Require}
  \SetKwInOut{KwEnsure}{Ensure}

  \KwRequire{Cluster distribution parameters $\lbrace\boldsymbol{\alpha}_{k}\rbrace_{k=0}^{K}$ (which serve as the true labels for generated data), $T$ (time steps), $R$ (repetitions), $N$ (cluster size), $m$ (distance function)}
  \KwEnsure{For function $m$, $K+1\ge 2$ clusters with $N$ members each, output an array of pairs $(t,\bar{s})$ for step $t$, mean Rand index $\bar{s}$ over $R$ clustering repetitions} 

  \caption{Steps to generate simulated data for testing different stochastic matrix distance functions. The Rand index is a measure of how well data has been clustered according to the true cluster label (which is the cluster distribution parameter $\alpha_{k}$)}.\label{proc:simulation}

  $t \leftarrow 0$\;
  \While{$t \le T$}{
    Update cluster parameters for time t\;
    $\alpha_k(t) \leftarrow (1 - t)\alpha_k + t\alpha_0$\; 
    $r \leftarrow 0$\;
    \While{$r < R$}{
      Generate $K \cdot N$ stochastic matrices based on parameters $\{\alpha_k(t)\}_{k=0}^K$\;
      Using function $m$, calculate the distance matrix for all pairs of the generated matrices\;
      Cluster the $K \cdot N$ matrices into $K$ labels using the distance matrix\;
      Calculate the Rand index $s(r)$ comparing the clustering labels to the known generative clusters\;
      $r \leftarrow r + 1$\;
    }
    Calculate the mean (adjusted) Rand index for time $t$: $\bar{s}(t) \leftarrow \frac{1}{R} \sum_{r=0}^{R-1} s(r)$\; 
    $t \leftarrow t + 1$\;
  }
\end{algorithm}

\begin{figure}
    \centering
    \includegraphics[width=1\linewidth]{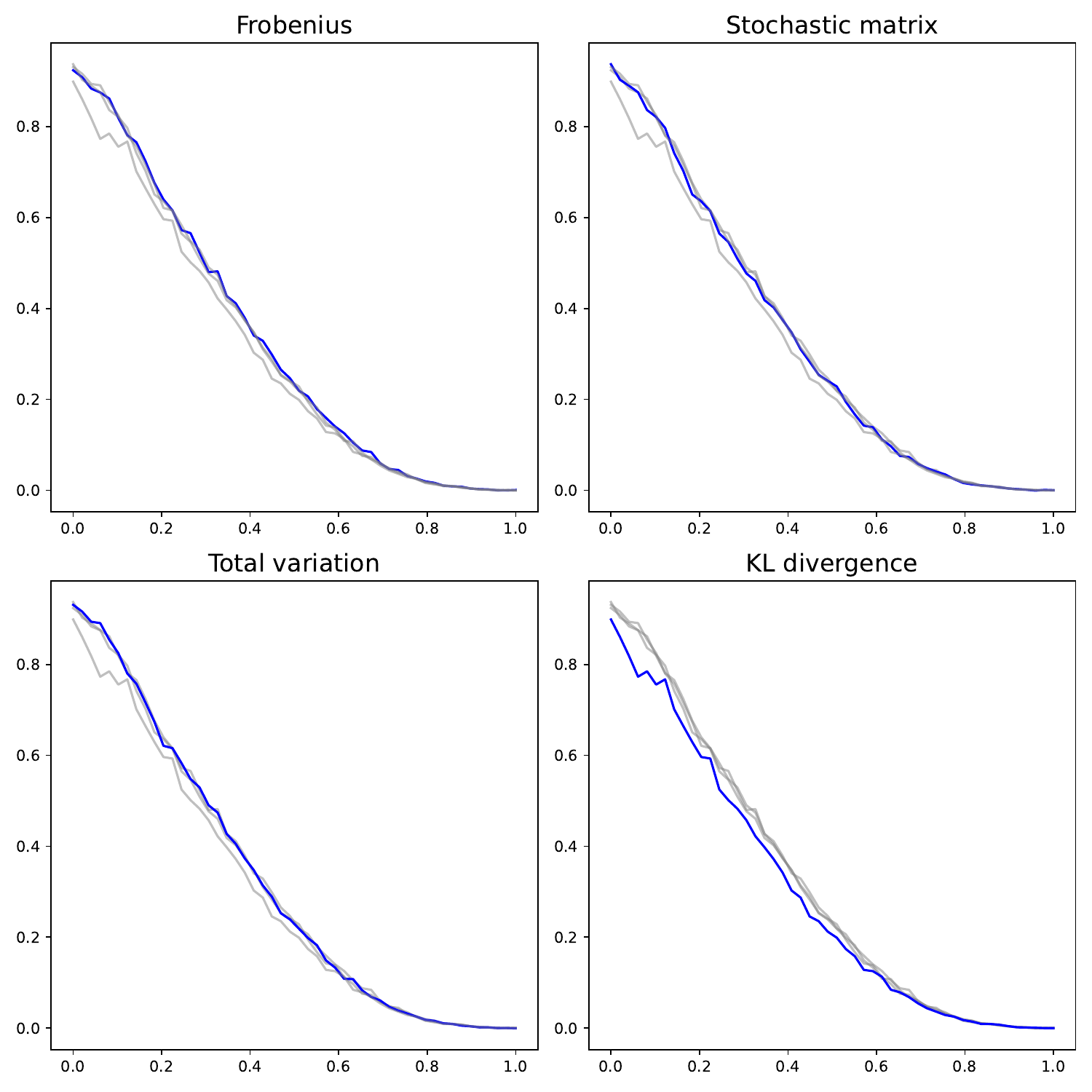}
    \caption{Comparison of simulation for the four selected distance functions. The vertical axis indicates the average (adjusted) Rand score attained for the clustering according to a given distance function. The horizontal axis indicates the step number (i.e. $t\in\lbrace 0,1/T,\ldots,1\rbrace$) starting from the initial distributions to the final distributions coinciding.}
    \label{fig:rand_index}
\end{figure}

\section{Conclusion}

We have developed the use of the Bhattacharyya angle as a distance function on discrete Markov chain sequences which is well motivated and respects the intrinsic probability properties of Markov chains. Previous work in discrete Markov theory had shown the use of various functions to compare discrete chains, but this work derives a true distance function which is computable (both analytically and numerically) and valid for any pair of chains. 

However these distances explicitly require the initial distributions of a chain and can be extremely sensitive to small changes. We therefore considered a distance on the space of stochastic matrices, which can give a more robust measure of similarity. Moreover, it is valid for any pair of stochastic matrices, is computable and through Eq.~\ref{eqn:ergodic_asymp_distance} in the case of ergodic chains the asymptotic step distance is the same as considering distances on the sequence space. We were also able to find compact analytical results for very broad classes of stochastic matrices, namely matrix products which are diagonalisable or primitive for sequences spaces, and ergodoic chains for the stochastic matrix distance. Using underlying local structures (i.e. state to state jumps) to understand Markov chain similarity in this sense is more akin to comparing Hamiltonians or Lagrangians in physics, rather than looking at the whole possible trace space dependent on initial distributions. 

Returning to our initial motivation, we wanted to find a way of comparing Markov chains for use in healthcare services. For example, if the emergency department of a hospital is described by a Markov chain, we can use our methods of comparing two hospitals. Moreover, each initial distribution of the Markov chain corresponds to a starting event for a patient. As there are many different patients, comparing the sequence space generated by all different patients across two hospitals doesn't have great meaning. What does have more meaning is the direct comparison of the stochastic matrices which represent the generative processes of the hospitals. Therefore our new stochastic matrix distance should have much applicability to comparing health services in any setting represented by Markov chains. Of course this if true for any other field in general.

Despite the number of results, there is still much more worthwhile future work. For sequence spaces, finding the most general conditions for the matrix $\boldsymbol{R}$ to give a closed form solution for $d^{\omega}(M_{1},M_{2})$ would be useful. We expect the omission of eigenvalues of $-1$ from the spectrum of $\boldsymbol{R}$ can be dropped, as we know the Bhattacharyya angle is always convergent. Also the exploration of mixing times and the application of the Bhattacharyya angle to general mixing time bounds has only received an elementary review here. Tighter bounds are surely possible. Applications of the Bhattacharyya angle, and its stochastic matrix counterpart, also seem readily available to machine learning in cases where the problem requires the evaluation of a model which outputs categorical probabilities. As both distances are both analytically and numerically approachable, it would be pleasant to see it incorporated in to computational frameworks, such as for gradient boosting~\cite{Chen2016XGBoost:System}. 

\begin{acknowledgments}
The authors would like to thank Sam Power, George Deligiannidis, Christopher Yau for their constructive and supportive comments for this paper. ARL would like to thank Sara Oriana Gomes Tavares for their very elegant argument on the convergence of the Bhattacharyya angle, which much improved upon his own. 

ARL acknowledges the receipt of studentship awards from the Health Data Research UK-The Alan Turing Institute Wellcome PhD Programme in Health Data Science (Grant Ref: 218529/Z/19/Z). P.
Ti\v{n}o was supported by the EPSRC Prosperity Partnerships
grant ARCANE, EP/X025454/1.
\end{acknowledgments}

\onecolumngrid 
\appendix

\section{Proof of theorem 2}

\begin{proof}
\label{proof:thm2}
    Using the product structure of stochastic matrices, we can decompose two stochastic matrices as
    \begin{equation}
        \begin{aligned}
            \boldsymbol{P}_{1} &= \left(\boldsymbol{p}_{1},\ldots,\boldsymbol{p}_{n}\right)\\
            \boldsymbol{P}_{2} &= \left(\boldsymbol{q}_{1},\ldots,\boldsymbol{q}_{n}\right)
        \end{aligned}
    \end{equation}
    Each $\boldsymbol{p}_{i}$ and $\boldsymbol{q}_{i}$ are probability vectors in their own right, with dimension $n$, and each matrix are therefore elements of the product space $\Delta\times\cdots\times\Delta$. This means we can construct the induced minimum length geodesic curve distance from each underlying simplex space. In other words,
    \begin{equation}
        d\left(\left(\boldsymbol{p}_{1},\ldots,\boldsymbol{p}_{n}\right), \left(\boldsymbol{q}_{1},\ldots,\boldsymbol{q}_{n}\right)\right) = \sqrt{d^{2}(\boldsymbol{p}_{1},\boldsymbol{q}_{1}) + \ldots + d^{2}(\boldsymbol{p}_{n},\boldsymbol{q}_{n})}
    \end{equation}
    As the underlying simplex spaces can be equipped with the Bhattacharyya angle, we can write the induced distance as
    \begin{equation}
        d\left(\left(\boldsymbol{p}_{1},\ldots,\boldsymbol{p}_{n}\right), \left(\boldsymbol{q}_{1},\ldots,\boldsymbol{q}_{n}\right)\right) = 2\sqrt{\sum_{i=1}^{n}\arccos^{2}\mathrm{BC}(\boldsymbol{p}_{i},\boldsymbol{q}_{i})}
    \end{equation}
   Noting $\boldsymbol{p}_{i}=(p_{i1},\ldots,p_{in})$, $\boldsymbol{q}_{1}=(q_{i1},\ldots,q_{in})$, where $p_{ij}=P_{1}(x_{i},x_{j})$ and $q_{ij}=P_{2}(x_{i},x_{j})$, it is easy to see the Bhattacharyya coefficient expressions can be written as $\mathrm{BC}(\boldsymbol{p}_{i},\boldsymbol{q}_{i}) = \sum_{j=1}^{n}\sqrt{p_{ij}q_{ij}}$ which allows us to write
    \begin{equation}
        d\left(\left(\boldsymbol{p}_{1},\ldots,\boldsymbol{p}_{n}\right), \left(\boldsymbol{q}_{1},\ldots,\boldsymbol{q}_{n}\right)\right) = 2\sqrt{\sum_{x_{i}\in X}\arccos^{2}\sum_{x_{j}\in X}\sqrt{P_{1}(x_{i},x_{j})P_{2}(x_{i},x_{j})}}
    \end{equation}
\end{proof}

% \section{A little more on appendixes}
% \subsection{\label{app:subsec}A subsection in an appendix}

\twocolumngrid
\bibliography{references}

\end{document}